\documentclass{article}
\title{Towards an Elementary Formulation of the Riemann Hypothesis in Terms of Permutation Groups}

\usepackage{amsthm} 
\usepackage{amsfonts}
\usepackage{amsmath}
\usepackage{mathabx}

\newtheorem{theorem}{Theorem}
\newtheorem*{thm*}{Theorem}

\newtheorem{lemma}{Lemma}[subsection]

\author{Will Cavendish and Jacob Tsimerman}

\numberwithin{mytheorem}{subsection}

\begin{document}
\maketitle

\begin{abstract}  This paper investigates the relationship between the Riemann hypothesis and the statement $\forall n, ~g(n) \le e^{\sqrt{p_n}}$, where $g(n)$ is the maximum order of an element of $S_n$, the symmetric group on $n$ elements, and $p_n$ is the $n$-th prime. We show this inequality holds under the Riemann Hypothesis. We also make progress towards establishing the converse by proving $\exists n,~g(n)>e^{\sqrt{p_n}}$ if the Riemann Hypothesis is false and the supremum of the set of the real parts of the Riemann zeta function's zeros $\sup \{\Re(\rho)~|~\zeta(\rho) = 0\}$ is not equal to 1. 
\end{abstract}

\noindent In a series of papers starting in the late 1960s (e.g., \cite{Nic1968}, \cite{MNR1988}, \cite{DN2019}), Nicolas and his collaborators established an intriguing relationship between the Riemann Hypothesis and the theory of permutation groups. One of the most striking, the central result of \cite{DN2019}, is that the Riemann Hypothesis is equivalent to the statement 
 \[
g(n)\le e^{\sqrt{\mathrm{li}^{-1}(n)}} \textrm{~ for all } n.
\]
Here $g:\mathbb{N}\to \mathbb{N}$ denote Landau's function, the function that takes $n\in \mathbb{N}$ to the maximum order of an element of $S_n$, the symmetric group on $n$ elements.

Note that by the prime number theorem, $\mathrm{li}^{-1}(n)$ is approximately equal to the $n$-th prime number $p_n$. This brings us to the central question of this paper: Can $\mathrm{li}^{-1}(n)$ be replaced by $p_n$ in the above equivalence? In what follows, we give a partial answer to this question.

\begin{theorem} \label{mainthm} If the Riemann Hypothesis is true, then 
\[
g(n)\le e^{\sqrt{p_n}}
\]
for all $n\ge 1$.
\end{theorem}

\begin{theorem}\label{thm2} Let $\zeta$ denote the Riemann zeta function. If the Riemann Hypothesis is false and $\sup \{\Re(\rho)~|~\zeta(\rho)=0\}\ne 1$, then there exists $n$ such that $g(n)> e^{\sqrt{p_n}}$.
\end{theorem}

\noindent These results show that question of whether the Riemann Hypothesis is equivalent to $\forall n,~g(n)\le e^{\sqrt{p_n}}$ comes down to the following unsolved problem: Does $g(n)$ take values greater than $e^{\sqrt{p_n}}$ in the case that the Riemann Hypothesis is ``as false as possible," i.e., when there exists a sequence $\{\rho_i\}$ of zeros of $\zeta(s)$  such that $\lim_{i\to \infty} \Re(\rho_i) = 1$? \newline

\noindent \textbf{Acknowledgements.} We are grateful to Jean-Louis Nicolas for his encouragement and helpful correspondence. We are also grateful to the referees for their feedback on earlier drafts of this paper.


\section{Bounding $g(n)$ under the Riemann Hypothesis}

We begin by establishing Theorem \ref{mainthm}. Following \cite{DN2019}, we define 

\begin{equation}\label{andefinition}
a_n=\frac{\sqrt{\mathrm{li}^{-1}(n)}-\log g(n)}{(n~\log n)^{1/4}}
\end{equation}

\noindent The central ingredient in our proof is the following result from \cite{DN2019}:

\begin{theorem}[Theorem 1.1(ii) in \cite{DN2019}]\label{DN2019} Under the Riemann Hypothesis, 
\[
a_n\ge \frac{2-\sqrt{2}}{3}-c-\frac{0.43 \log \log n}{\log n}~\textrm{for all }n\ge 2
\]
where $c=\sum_{\rho} \frac{1}{|\rho(\rho+1)|} \approx 0.046117644421509...$  and the sum is taken over the set of non-trivial zeros of the Riemann $\zeta$ function. 
\end{theorem}


\noindent The second estimate we will need is the following:

\begin{lemma}\label{firstlemma} Under the Riemann Hypothesis,
$$\left|\mathrm{li}^{-1}(n)-p_n\right|\le  \frac{\sqrt{2}}{8\pi}\log^2(2n\log n)\sqrt{n\log n}$$
for all $n\ge 2657$.
\end{lemma}

\begin{proof}

Under the Riemann Hypothesis, a well-known result of Schoenfeld \cite[Corollary 1]{Sch1976} gives

\[
|\pi(x)-\mathrm{li}(x)|<\frac{\sqrt{x}\log x}{8\pi}
\]

\noindent for all $x\ge 2657$ where $\pi(x)$ denotes the prime-counting function. Plugging in the $n$-th prime $p_n$ and observing that $\pi(p_n)=n$ we get:

\begin{equation}\label{distnlip}
|n-\mathrm{li}(p_n)|<\frac{\sqrt{p_n}\log p_n}{8\pi}.
\end{equation}

\noindent Applying the mean value theorem to $\mathrm{li}(x)$ at the points $p_n$ and $\mathrm{li}^{-1}(n)$, we have
\[
\frac{\mathrm{li}(\mathrm{li}^{-1}(n))-\mathrm{li}(p_n)}{\mathrm{li}^{-1}(n)-p_n}=\mathrm{li}'(x_n)=\frac{1}{\log x_n}
\]

\noindent for some $x_n$ between $\mathrm{li}^{-1}(n)$ and $p_n$. Taking the absolute value of both sides and multiplying through by denominators, we have

\begin{equation}\label{eq2}
|\mathrm{li}^{-1}(n)-p_n|=(\log x_n)|n-\mathrm{li}(p_n))|
\end{equation}

\noindent Combining this with equation (\ref{distnlip}), we have

\begin{equation}\label{2}
|\mathrm{li}^{-1}(n)-p_n|<(\log x_n)\frac{\sqrt{p_n}}{8\pi}\log p_n
\end{equation}

\noindent for $n\ge 2657$. An elementary argument shows that  $t\mapsto \mathrm{li}(2t \log t)-t$ is positive for $t \ge 3$. Since $\mathrm{li}^{-1}(t)$ is strictly increasing on  $(1, \infty)$, this implies that $\mathrm{li}^{-1}(n) < 2n \log n$ for $n\ge 3$.

From \cite[(3.13)]{RS}, for $n\ge 6$,

\[
p_n < n(\log n + \log \log n) < 2n \log n
\]

\noindent Since $x_n$ lies between $\mathrm{li}^{-1}(n)$ and $p_n$, $x_n<2n\log n$ for all $n\ge 3$ as well. Substituting these inequalities into inequality (\ref{2}) above, we get

\[
|\mathrm{li}^{-1}(n)-p_n|<\log(2n\log n)\left|\frac{1}{8\pi}\sqrt{2n\log n}\log(2n\log n))\right|
\]

$$= \frac{\sqrt{2}}{8\pi}\log^2(2n\log n)\sqrt{n \log n}$$

\noindent for all $n\ge 2657$.

\end{proof}


\begin{lemma}\label{RHmainineq} Under the Riemann Hypothesis,
\[
\sqrt{\mathrm{li}^{-1}(n)}-\sqrt{p_n}<a_n \cdot (n\log n)^{1/4}
\]
for all $n> 10^{10}$.
\end{lemma} 

\begin{proof} By the mean value theorem applied to the function $x\mapsto \sqrt{x}$ at the points $\mathrm{li}^{-1}(n)$ and $p_n$, 

\begin{equation}\label{mvt2}
\sqrt{\mathrm{li}^{-1}(n)}-\sqrt{p_n}= \frac{\mathrm{li}^{-1}(n)-p_n}{2\sqrt{x_n}} 
\end{equation}

\noindent for some $x_n$ lying between $\mathrm{li}^{-1}(n)$ and $p_n$. By Rosser's theorem \cite{Ros1939}, $p_n>n\log n$ for all $n$. A simple calculation shows that the mapping $t\mapsto t - \mathrm{li}(t\log t)$ is increasing for $t>e^e$ and positive for $t > 40.5 $ and therefore $\mathrm{li}^{-1}(n)>n\log n$ for all integers $n>40$. Since $x_n$ lies between $p_n$ and $\mathrm{li}^{-1}(n)$, it follows that $x_n>n\log n$, so

$$\sqrt{n\log n} < \sqrt{x_n}$$

\noindent for $n>40$. We therefore have

$$\sqrt{\mathrm{li}^{-1}(n)}-\sqrt{p_n}<\frac{|\mathrm{li}^{-1}(n)-p_n|}{2\sqrt{n\log n}}$$

\noindent for $n>40$ by equation (\ref{mvt2}).

Applying Lemma \ref{firstlemma} to the numerator of the right-hand side of the above, we get 

\begin{equation}\label{lipnbound} 
\sqrt{\mathrm{li}^{-1}(n)}-\sqrt{p_n}< \frac{\sqrt{2}}{16\pi}\log^2(2n\log n)
\end{equation}

\noindent for all $n\ge 2657$.

Using Theorem \ref{DN2019}, we have

\[
 0.14 - \frac{0.43 \log \log n}{\log n} \le \frac{2-\sqrt{2}}{3}-c-\frac{0.43 \log \log n}{\log n} \le a_n
\]

\noindent for all $n\ge 2$. A simple calculation also shows that the left-hand side of the above equation is always larger than 0.08 for $n> 10^{10}$, so

\begin{equation}\label{ineq2}
0.08\cdot (n\log n)^{1/4}< a_n\cdot (n\log n)^{1/4}
\end{equation}
for all $n> 10^{10}$.

Direct calculation also gives

\[
\frac{\sqrt{2}}{16\pi}\log^2(2n\log n)<0.08\cdot (n\log n)^{1/4}
\]

\noindent for all $n>10^{10}$, so combining the above with equation (\ref{ineq2}) we get

\[
\frac{\sqrt{2}}{16\pi}\log^2(2n\log n)<a_n\cdot (n\log n)^{1/4}
\]

\noindent for all $n>10^{10}$. Putting this together with equation (\ref{lipnbound}) above, we get

\[
\sqrt{\mathrm{li}^{-1}(n)}-\sqrt{p_n}<a_n\cdot (n\log n)^{1/4}
\]
\noindent for all $n>10^{10}$.

\end{proof}

\noindent We are now ready to prove Theorem \ref{mainthm}.

\begin{proof}[Proof of Theorem \ref{mainthm}] By taking the logarithm of both sides of the inequality $g(n)\le e^{\sqrt{p_n}}$ and rearranging terms, we obtain the inequality 
\[
\sqrt{p_n}-\log(g(n))>0.
\]
It therefore suffices to show that, under the Riemann Hypothesis, this inequality holds for all $n \ge 1$. We proceed by showing this in two cases, one for $n>10^{10}$, and one for $n\le 10^{10}$.

For $n>10^{10}$, Lemma \ref{RHmainineq} together with the definition of $a_n$ gives us

\[
\sqrt{p_n}-\log(g(n))=\sqrt{\mathrm{li}^{-1}(n)}-\log(g(n))-(\sqrt{\mathrm{li}^{-1}(n)}-\sqrt{p_n})
\]
\[
>\sqrt{\mathrm{li}^{-1}(n)}-\log(g(n))-a_n \cdot (n\log n)^{1/4}=0
\]

\noindent We now consider the case where $n\le10^{10}$. For $n=1$ or 2, one can easily check that $g(n)\le e^{\sqrt{p_n}}$. For $3\le n\le10^{10}$, $p_n \le 2 n \log n \le 2 \times 10^{10} \log(10^{10}) < 10^{14}$ holds. If $m\le 10^{14}$, $ \pi(m) < \mathrm{li}(m)$ by a result of Kotnik \cite{Kotnik} (that has subsequently been proven for all $m$ up to $10^{19}$ by B\" uthe \cite{Buthe}.) Therefore \newline $n = \pi(p_n) < \mathrm{li}(p_n)$, from which it follows that $p_n > \mathrm{li}^{-1}(n)$.

Applying the square roots to both sides of this inequality, we have \\ $\sqrt{p_n}>\sqrt{\mathrm{li}^{-1}(n)}$, so
\[
\sqrt{p_n}-\log(g(n))> \sqrt{\mathrm{li}^{-1}(n)}-\log(g(n))
\]
 for all  $n<10^{10}$.
 
Since $\sqrt{\mathrm{li}^{-1}(n)}>\log(g(n))$ under the Riemann Hypothesis by the central result of \cite{DN2019}, the righthand side of the above is always positive, so 
\[
\sqrt{p_n}-\log(g(n))>0
\]
as required.

\end{proof}


\section{Finding large values of $g(n)$ when the Riemann Hypothesis is false}

\noindent We now turn to the proof of the second part of Theorem \ref{mainthm}. Throughout this section, we let 
$$\Theta=\sup \{\Re(s)~|~ \zeta(s)=0\}$$ 
and we assume $\Theta>\frac{1}{2}$.

Our approach follows the work of Massias, Nicolas, and Robin in \cite{MNR1988}. We begin by recalling some terminology and results from \cite{MNR1988}.  Building on the work of Nicolas in \cite{Nic1968}, the authors of \cite{MNR1988} define a non-decreasing function $\mathbb{R}^+\to \mathbb{N}$ mapping $\eta \mapsto N_\eta$ with the property $N_\eta \in g(\mathbb{N})$ for all $\eta$. Given any $n\in \mathbb{N}$, they define $\rho=\rho(n)$ to be such that 
$$N_\rho =  \max\{N_\eta~|~N_\eta \le g(n)\}.$$ 
They also define $x_1=x_1(n)$ be such that 
$$x_1/\log x_1 = \rho(n).$$ 
We note that $x_1(n)$ and $\rho(n)$ are non-decreasing functions of $n$. Following Nicolas et al., we will leave the dependence of $x_1$ and $\rho$ on $n$ implicit in many of the expressions that follow.

The proof of Theorem \ref{thm2} will require several bounds related to $\rho$, $x_1$ and $N_\rho$ established in \cite{MNR1988}. The first is given by the following lemma, which appears as equation (6) in \cite{MNR1988}. Let $\theta$ and $\psi$ denote Chebyshev's functions  
$$\theta(x)=\sum_{p \le x} \log p$$
and
$${\psi(x)=\sum_{p^k \le x} \log p}.$$

\begin{lemma}\label{logginequality} For $x_1$ and $N_{\rho}$ as above, we have
$$\theta(x_1)\le \log N_\rho\le \psi(x_1).$$
\end{lemma}

\noindent Given a real-valued function $f$ and a positive function $h$, let $f(x) = \Omega_+(h(x))$ denote $\displaystyle{\limsup_{x\to \infty}} \frac{f(x)}{h(x)}>0$. The following lemma follows directly by combining results of \cite{Nic1968}.

\begin{lemma}\label{nufunction} If $\frac{1}{2}<\Theta<1$,
\[
\mathrm{li}(\log^2(g(n))-n = \Omega_{+}\left(x_1^{\Theta+1}/\log x_1\right)
\]
\end{lemma}

\begin{proof} Equation (28) in \cite{MNR1988} gives
\begin{equation}\label{eq28}
\mathrm{li}(\log^2(g(n)))-n = -\left(\Pi_1(x_1)-\mathrm{li}(\psi^2(x_1))\right)+O(x_1^{3/2}/\log x_1)
\end{equation}

\noindent where $\Pi_1(x)=\sum_{p^k \le x} \frac{p^k}{k}$ and $\psi$ is Chebyshev's function $\psi(x)=\sum_{p^k \le x} \log p$. As is pointed out in \cite{MNR1988}, the convexity of the function $t\mapsto \mathrm{li}(t^2)$ for $t \ge e$ implies
\[
\mathrm{li}(\psi^2(x_1))\ge \mathrm{li}(x_1^2)+\frac{x_1}{\log x_1} (\psi(x_1)-x_1)
\]
\noindent for all sufficiently large $x_1$. Substituting this into equation (\ref{eq28}) and rearranging terms, we have
\begin{equation}\label{eqw}
\mathrm{li}(\log^2(g(n))-n \ge \mathrm{li}(x_1^2) - \Pi_1(x_1)+\frac{x_1}{\log x_1} (\psi(x_1)-x_1) +O\left(x_1^{3/2}/\log x_1\right). 
\end{equation}

\noindent From Lemma C part (iii) of \cite{MNR1988}, if $\Theta<1$,
\[
\mathrm{li}(x_1^2)-\Pi_1(x_1)+\frac{x_1}{\log x_1} (\psi(x_1)-x_1)=\Omega_{+}(x_1^{\Theta+1}/\log x_1).
\]
Substituting this into equation (\ref{eqw}), we have
\[
\mathrm{li}(\log^2(g(n))-n \ge \Omega_{+}(x_1^{\Theta+1}/\log x_1) +O\left(x_1^{3/2}/\log x_1\right). 
\]
Since $\Theta+1>3/2$ by assumption, $x_1^{\Theta+1}/\log x_1$ dominates the $O\left(x_1^{3/2}/\log x_1\right)$ term, so we obtain 
\[
\mathrm{li}(\log^2(g(n))-n = \Omega_{+}(x_1^{\Theta+1}/\log x_1)
\]
as required.
\end{proof}

\noindent A third result we will need is that any element of the image $g(\mathbb{N})$ of Landau's function is close to an element of the image of $\rho\mapsto N_\rho$ as given by the following lemma (equation (11) from \cite{MNR1988}):

\begin{lemma}\label{nearestnrho} With $N_\rho$ as above, 
\[
\log(g(n)) =\log N_\rho+O(\log x_1).
\]
\end{lemma}

\noindent The final ingredients we will need are bounds on the error term in the prime number theorem. To this end, we consider the function
\[
R(x)= \sup_{e \le s \le x} |\pi(s)- \mathrm{li}(s)|,
\]
which, as we will show, satisfies the following bound.

\begin{lemma}\label{Rissublinear}  For all $x\in[e, \infty)$ and $b \in \mathbb{R}^+$, 
$R(x+b) \le R(x) + 2(b+ e + 1)$
\end{lemma}

\noindent The proof of Lemma \ref{Rissublinear} requires the following elementary lemma about the growth of the absolute value of the difference between two positive, monotone increasing functions that satisfy a sublinearity condition.

\begin{lemma}\label{sublinfixed}
Let $L\ge 0$ and let $f_1$ and $f_2$ be positive, monotone increasing functions such that for all $x\in [L,\infty)$ and $i\in \{1,2\}$
\begin{itemize}
\item $f_i(x) \le x$
\item there exists a constant $C>0$ such that for all $b\in \mathbb{R}^+$, \\ $f_i(x+b) - f_i(x)\le b+C$
\end{itemize}
Then $h(x) =\displaystyle \sup_{L\le s \le x }|f_1(s)-f_2(s)|$ satisfies 
\[
h(x+b) \le h(x) + 2(b + C + L )
\]
for all $x\in [L,\infty)$ and $b\in \mathbb{R}^+$.

\end{lemma}

\begin{proof}

Let $k(s) = |f_1(s)-f_2(s)|$, so $h(x)=\displaystyle \sup_{L\le s\le x} k(s)$. Since
\[
h(x+b) =  \sup_{L\le s\le x+b} k(s) = \max\left\{\sup_{L\le s\le L+b} k(s) ,\sup_{L+b \le s \le x+b} k(s) \right\},
\]
it suffices to prove
\begin{equation}\label{ineqeasy}
\sup_{L\le s\le L+b} k(s) \le  h(x) + 2(b + C + L )
\end{equation}
and
\begin{equation}\label{ineqhard}
\sup_{L+b \le s \le x+b} k(s) \le  h(x) + 2(b + C + L ).
\end{equation}

\noindent Inequality (\ref{ineqeasy}) follows from the sublinearity assumptions on $f_i$, since
\[
\sup_{L\le s \le  L+b} k(s) =  \sup_{L\le s \le L+b} |f_1(s)- f_2(s)| \le  \sup_{L\le s \le L+b} |f_1(s)| + \sup_{L\le s \le L+b} |f_2(s)|
\]
\[
\le \sup_{L\le s \le L+b} s + \sup_{L\le s \le L+b} s 
= 2(L+b) \le h(x) + 2(b+C+ L).
\]
\noindent For inequality (\ref{ineqhard}), the positivity and monotonicity of $f_i$ together with the bound $f_i(x+b) - f_i(x)\le b+C$ give us
\[
k(s+b) = |f_1(s+b) -f_2(s+b)| 
\]
\[
= |f_1(s+b) -f_1(s) - (f_2(s+b) - f_2(s)) + f_1(s) -f_2(s) |
\]
\[
\le  |f_1(s+b) -f_1(s)| + |f_2(s+b) - f_2(s)| + |f_1(s) -f_2(s)|
\]
\[
= f_1(s+b) -f_1(s)~+~f_2(s+b) - f_2(s)~ +~ k(s)
\]
\[
\le 2(b+C) +  k(s).
\]
\noindent Taking the supremum over $s\in [L , x]$ of both sides of this inequality, we have 
\[
\sup_{L\le s \le x} k(s+b) \le \sup_{L\le s\le x} k(s) +2(b+C) = h(x) + 2(b+C) \newline \le h(x)+ 2(b+C+L)
\]
\noindent Since 
\[
 \sup_{L\le s \le x} k(s+b) = \sup_{L+b\le s \le x+b} k(s) 
\]
this establishes inequality (\ref{ineqhard}).
\end{proof}

\noindent We now use Lemma \ref{sublinfixed} to prove Lemma \ref{Rissublinear}


\begin{proof}[Proof of Lemma \ref{Rissublinear}] By Lemma \ref{sublinfixed}, it suffices to check that $\mathrm{li}(x)$ and $\pi(x)$ satisfy the assumptions of the lemma on $[e,\infty)$ with constants $C=1$ and $L=e$. Monotonicity and positivity follow easily from the definition, as does $\pi(x)\le x$ and $\mathrm{li}(x)\le x$ for $x\ge 2$.

To see that $\pi(x+b)\le \pi(x) + b+1$, note that $\pi(x+b)-\pi(x)$ counts the number of primes in the interval $(x, x+b]$. Since this interval contains at most $b+1$ integers, the result follows.

\noindent For $\mathrm{li}(x)$, we have that for all $x\ge e$ and $b\ge 0$,  

\[
\mathrm{li}(x+b) = \mathrm{li}(x) + \int_x^{x+b} \frac{ds}{\log s}
\]
\[
\le \mathrm{li}(x) + \int_x^{x+b} 1\cdot ds =  \mathrm{li}(x)+b \le \mathrm{li}(x) + b+ 1
\]

\noindent since $1/\log s <1$ on $[e,\infty)$.

\end{proof}


\begin{lemma}\label{errorbound} If $1/2<\Theta<1$, then
$R(\log^2(g(n)))=O(x_1^{2\Theta} \log x_1).$
\end{lemma}

\begin{proof}

\noindent Let $N_\rho$ be the largest element of the image of the map $\rho\mapsto N_\rho$ less than $g(n)$ as above. By Lemma \ref{nearestnrho},
\[
\log(g(n))\le \log N_\rho + C\log x_1
\]
\noindent for some $C>0$, so 
\[
\log^2(g(n))\le \log^2(N_\rho) + 2C(\log N_\rho)(\log x_1)+C^2\log^2(x_1).
\]
\noindent By the monotonicity of $R$ and Lemma \ref{Rissublinear}, we have 
\[
R(\log^2(g(n)))\le R( \log^2(N_\rho) + 2C(\log N_\rho)(\log x_1)+C^2\log^2(x_1))
\]
\[
 \le R( \log^2(N_\rho)) + 2(2C(\log N_\rho)(\log x_1)+C^2\log^2(x_1)+e+1).
\]
\noindent By Lemma \ref{logginequality}, we have that the above is bounded by
\[
R(\psi^2(x_1)) + 4C\psi(x_1)\log x_1+2C^2\log^2(x_1)+ 2e + 2 = R(\psi^2(x_1))+O(x_1 \log x_1)
\]
\noindent since $\psi(x_1)=O(x_1)$ by the prime number theorem. Since $x_1\log x_1$ is negligible relative to $x_1^{2\Theta}\log x_1$ when $\Theta>1/2$, it therefore suffices to show that $R(\psi^2(x_1))=O(x_1^{2\Theta}\log x_1)$.

Since $\Theta<1$ by assumption, we have that $R(x_1) = O(x_1^\Theta \log x_1)$ (cf \cite{In1990} Theorem 30). Using $\psi(x_1)=O(x_1)$ once again,
\[
R(\psi^2(x_1)) = O(R(x_1^2))= O(x_1^{2\Theta} \log x_1)
\]
\noindent as required.

\end{proof}

\noindent We are now ready to prove Theorem \ref{thm2}, i.e., if $\Theta>\frac{1}{2}$, there exists an integer $n$ such that $g(n)> e^{\sqrt{p_n}}$. 

\begin{proof}[Proof of Theorem \ref{thm2}]

\noindent Since the exponential, square-root, and prime counting function $\pi$ are all monotone, the conclusion of Theorem \ref{thm2} is equivalent to the statement  
\[
\exists n, ~\pi(\log^2(g(n)))-n>0.
\]

\noindent Let $N_\rho$ be the largest element of the image of the map $\rho\mapsto N_\rho$ less than $g(n)$. Then
\[
\pi(\log^2(g(n)))-n =\mathrm{li}(\log^2(g(n)))-n+\pi(\log^2(g(n)))-\mathrm{li}(\log^2(g(n)))
\]
\begin{equation}\label{eq1}
\ge \mathrm{li}(\log^2(g(n)))-n- R(\log^2(g(n))).
\end{equation}
\noindent By Lemma \ref{nufunction}, we have
\[
\mathrm{li}(\log^2(g(n)))-n = \Omega_{+} (x_1^{\Theta+1}/\log x_1)
\]
\noindent and by Lemma \ref{errorbound}, we have 
\[
R(\log^2(g(n)))=O(x_1^{2\Theta} \log x_1).
\]
\noindent Applying this to equation (\ref{eq1}) we have
\[
\pi(\log^2(g(n)))-n \ge  \Omega_{+} (x_1^{\Theta+1}/\log x_1) - O(x_1^{2\Theta} \log x_1).
\]
\noindent Since $\Theta<1$, $2\Theta< \Theta+1$, so $R(\psi^2(x_1))$ is negligible relative to $x_1^{\Theta+1}/\log x_1$ and therefore 
\[
\pi(\log^2(g(n)))-n \ge  \Omega_{+} (x_1^{\Theta+1}/\log x_1).
\]
\noindent It follows that $\pi(\log^2(g(n))-n$ must take a positive value for some $n$.

\end{proof}


\end{document}